\documentclass[letterpaper, 10 pt, journal,twoside]{IEEEtran}
\IEEEoverridecommandlockouts                              

\title{\LARGE \bf
On the mean-field Belavkin filtering  equation}

\author{Sofiane Chalal$^{1}$,  Nina H. Amini$^{1}$, Gaoyue Guo $^{2}$ 
\thanks{This work is supported by the Agence Nationale de la Recherche projects Q-COAST ANR- 19-CE48-0003 and IGNITION ANR-21-CE47- 0015. }
\thanks{$^{1}$ CNRS, L2S, CentraleSupélec, Université Paris-Saclay.
        {\tt\tiny firstname.lastname@centralesupelec.fr}.}%
\thanks{$^{2}$ MICS, CentraleSupélec, Université Paris-Saclay.
        {\tt\tiny firstname.lastname@centralesupelec.fr}.}%
}

\usepackage{xcolor} 
\usepackage{amssymb} 
\usepackage{amsmath} 
\usepackage{bigints} 
\usepackage{braket} 
\usepackage{physics} 
\usepackage{graphicx} 
\usepackage[nottoc]{tocbibind} 
\usepackage{url}
\newtheorem{theorem}{Theorem}

\newtheorem{lemma}{Lemma}

\newtheorem{remark}{Remark}

\newcommand{\iu}{\mathrm{i}\mkern1mu} 
\newcommand{\uc}{\boldsymbol{\mathrm{u}}} 

\def \C{\mathbb{C}}

\def \E{\mathbb{E}}

\def \H{\mathbb{H}}

\def \N{\mathbb{N}}

\def \R{\mathbb{R}}

\def\Uc{{\cal U}}

\def\Xc{{\cal X}}

\def \d {\delta}

\def \hm {\hat{\mu}}

\def \d{{\rm d}}
\def \1{\mathds{1}}

\begin{document}

\maketitle
\thispagestyle{empty}
\pagestyle{empty}

\begin{abstract}
Following Kolokoltsov's work \cite{kolokoltsov2022qmfg}, we present an extension of mean-field control theory in quantum framework. In particular such an extension is done naturally by considering the Belavkin quantum filtering and control theory in a mean-field setting. In this setting, the dynamics is described by a controlled Belavkin equation of McKean-Vlasov type. We prove the well-posedness of such an equation under imperfect measurement records. Furthermore, we show under purification assumption the propagation of chaos for perfect measurements. Finally, we  apply particle methods to simulate the mean-field Belavkin equation and we provide numerical simulations showing the stabilization of the mean-field Belavkin equation by a feedback control strategy  towards a chosen target state.

\textit{Keywords : }Quantum filtering, Stochastic control, Mean-field Belavkin equation, Quantum state reduction, Stabilization in mean-field. 

\end{abstract}

\section{Introduction}

 Mean-field (MF) game  theory, lying at  the intersection of game theory and stochastic control theory, is the study of strategic decision made by interacting indistinguishable agents in very large populations. This class of problems was considered in the engineering literature by Huang,  Malhame and Caines \cite{huang2006large} and independently and around the same time by mathematicians Lasry and  Lions \cite{lasry2006mfg1}. Namely, consider $N$ agents whose states evolve according to the stochastic differential equations below: for $j=1,\ldots, N$, 
\begin{align*}
\mathrm{d}X^{\uc,j}_t &= b(X^{\uc,j}_t,u^{j}_t,\hm_t^{\uc, N})\mathrm{d}t + \sigma(X^{\uc,j}_t,u^{j}_t,\hm_t^{\uc, N})\mathrm{d}W_t^{j} \\
\hm_t^{\uc, N}&:=\frac{1}{N}\sum_{k=1}^{N}\delta_{X^{\uc,k}_t},\;\;
\hm^{\uc, N}=(\hm_t^{\uc, N})_{0\leq t\leq T},
\end{align*}
where $T$ is supposed as the final time, and $b, \sigma$ are suitable  functions and $W^1,\ldots, W^N$ are independent Brownian motions. Here $X^{\uc,j}_t$ stands for the state of agent $j$ at time $t$ subject to strategy profile $(u^1,\ldots, u^N)=:\uc$, and each agent interacts with the others through the empirical measure $\hm_t^{\uc, N}$. Provided some set $\Uc$ of admissible  strategies and a time horizon $T$, agent $j$ aims to minimize  its cost $\Uc\ni u\mapsto \mathcal{J}_j(u)\in\R$ with $\mathcal{J}_j(u)\equiv \mathcal{J}(u, \hm^{\uc_j,N})$:  
$$\mathcal{J}_j(u):= \mathbb{E}\left[\int_{0}^{T} f(X^{\uc_j,j}_t,u_t,\hm_t^{\uc_j,N})\d t + g(X_T^{\uc_j, j},\hm_T^{\uc_j,N})\right],$$
where $\uc_j:= (u^1,\ldots, u^{j-1}, u, u^{j+1},\ldots, u^N)$  and $f, g$ are some cost  functions. Nash equilibrium, where no player can do better by unilaterally changing their strategy,  is the most common way to define the solution of such a non-cooperative game. Namely, $\uc^*:=(u^{*,1},\ldots, u^{*,N})\in\Uc^{{N}}$ is said to achieve a Nash equilibrium if 
$$\mathcal{J}(u^{*,j},\hm^{\uc^{*},N})= \inf_{u\in\Uc} \mathcal{J}(u,\hm^{\uc^*_j,N}),$$ where $\uc_j^*:= (u^{*,1},\ldots, u^{*,j-1}, u, u^{*,j+1},\ldots, u^{*,N}).$ 
Generally there is no explicit expression for the Nash equilibrium, and its numerical computation is quite costly. Given  the importance for applications, as well as its active theoretical interest, it becomes increasingly important to consider the MF limit as $N\to\infty$. Hence, the corresponding MF game consists of finding a pair  $(\hat u, \hat X)$ satisfying
\begin{align*}
\mathrm{d}\hat{X}_t = b(\hat{X}_t,\hat u_t,& \mathcal{L}(\hat{X}_t))\mathrm{d}t + \sigma(\hat{X}_t,\hat u_t, \mathcal{L}(\hat{X}_t))\mathrm{d}W_t,\\
&\mathcal{J}(\hat u,\hm) \leq \mathcal{J}(u,\hm),
\end{align*}
where $\hm:=(\mathcal{L}(\hat X_t))_{0\leq t\leq T},$ and $\mathcal{L}$ denotes the law of random variable $\hat{X}_t$.

The empirical measure playing a major role in classical MF game, does not have an analogue in quantum setting, since $N-$particle quantum evolution particles are not separated in individual dynamics due to entanglement between particles. The other difficulty is related to quantum measurements which perturb the state of the system, which is known as a back-action effect. Moreover, measuring continuously freezes the  dynamics of the system \cite{sudarshan77}. Hence a new methodology is required to build a quantum MF game theory.

In a remarkable series of papers \cite{kolokoltsov2021law,kolokoltsov2022dynamic,kolokoltsov2022qmfg,kolokoltsov2021qmfgcounting},  Kolokoltsov has developed a  methodology for quantum MF games, where indirect measurements are considered to conserve the system's dynamics. In this framework, the dynamics is described by matrix-valued stochastic differential equations. As same as for the classical case, the propagation of chaos has been derived by adopting the approach of Pickl \cite{pickl11simple} to a stochastic version. 
It should also be noted that this new framework allows us to deal with a measurement-based feedback control problem of quantum systems with high dimensionality. Feedback control of  quantum systems plays a major role in controlling quantum systems in a robust fashion, see e.g.,  \cite{serafini12feedback,gough13,handel05rev,wiseman2009quantum}. Due to high dimensionality of the system, realization of a feedback control in real-time is time-consuming and not practical in a real experiment. 


In the following, we recall Belavkin quantum filtering theory \cite{boutenhandel07,gough22,ohki18,belavkin01QuantumNB} and we discuss the extension of  classical MF games and control characteristics in the quantum filtering framework, which is proposed in \cite{kolokoltsov2022qmfg}. The quantum filtering framework represents a natural one to construct a quantum MF game theory.
Later, motivated by games with incomplete information, we extend the MF Belavkin equation in the case of imperfect measurement records, and give a proof of the well-posedness of such the equation. Furthermore, for perfect measurements, we show the propagation of chaos under the purification assumption, i.e., asymptotically the mixed states become pure states, see e.g., \cite{maassen2006purification}.
Finally,  we use particle methods algorithm to simulate the MF equation.  We suggest the use of quantum MF filtering as a method to reduce the complexity representation of open quantum systems, which is usually high. In the case of quantum non-demolition measurement,  simulations illustrate a quantum state reduction (see e.g., \cite{handel05red,bauer2013repeated,bauer2011convergence,liang2019exponential}) for MF particles. This is encouraging to apply such a method, for instance, in feedback stabilization based on such an MF theory. Inspired by \cite{liang2018exponential}, we construct a control law depending on the MF equation, through simulations, we observe stabilization of the system toward the target state.
\subsection{Preliminaries}
We fix throughout the paper a finite set $\Xc = \{1,\dots,d\}$ and set $\mathbb{H} := \C^{d}.$
Let $M_d$ be the set of $d\times d$ complex matrices. For every $A\in M_d$, denote by $A^\dag$ its conjugate transpose. Define further the set of density matrices ${S}_{d}:=\{ \rho \in {M}_{d} :~ \rho = \rho^{\dag}, ~ \rho \geq 0,~   
 tr(\rho) = 1 \}$. For any   $A,B\in M_{d}$, set $[A,B] := AB - BA$  and $\{A,B\} := AB + BA$. For every $N\in\N$, let $\mathbb{H}^{\otimes {N}}$ denote the $N-$tensor product  of $\H$. 

For any operator $B:\H\to\H$ and for $j=1,\ldots,N,$ denote by ${\bf B}_j :\mathbb{H}^{\otimes {N}} \to\mathbb{H}^{\otimes {N}}$ the operator acting only on the sub-system living on the $j$-th Hilbert space $\H$, i.e. ${\bf B}_j(h_1\otimes\cdots\otimes h_j\otimes\cdots\otimes h_N):=(h_1\otimes\cdots \otimes B(h_j)\otimes\cdots\otimes h_N)$. Similarly for any operator $O:\H\otimes\H\to\H\otimes\H$, i.e. $O(\cdot\otimes\cdot):=(O_1(\cdot\otimes\cdot)\otimes O_2(\cdot\otimes\cdot))$, and for $j\neq k\in \{1,\ldots,N\}$ denote by ${\bf O}_{jk} :\mathbb{H}^{\otimes {N}} \to\mathbb{H}^{\otimes {N}}$ the operator acting only on the sub-systems living on the product of $j$-th and $k$-th Hilbert spaces $\H$, i.e. ${\bf O}_{jk}(h_1\otimes\cdots\otimes h_j\otimes\cdots\otimes h_k\otimes\cdots\otimes h_N):=(h_1\otimes\cdots\otimes O_1(h_j\otimes h_k)\otimes\cdots\otimes O_2(h_j\otimes h_k)\otimes\cdots\otimes h_N)$. 


\section{Quantum filtering and control}
Having examined the characteristics in the classical case, we want to extend them to the case where particles obey the principles of quantum mechanics.

In a dynamic game situation with $N$-players, the strategies are made in real-time, and therefore the system must be measured continuously.  The quantum system to consider is therefore necessarily open, in order to observe the evolution of the state and to avoid quantum Zenon effect, we have to pass through indirect measurements \cite[Section 4                 
]{belavkin92}. The control induced by each player is done via a controlled electromagnetic field.


An open quantum system undergoing continuous-time measurements 
can be described mathematically by a matrix-valued stochastic differential equation called Belavkin quantum filtering equation
\begin{align*}
\mathrm{d}\rho_t&= \left(-\iu[{H} + u(\rho_t)\hat{H},\rho_t] + \big( L\rho_tL^{\dag} - \frac{1}{2}\big\{L^{\dag}L,\rho_t\big\}\right)\mathrm{d}t \\
+& \sqrt{\eta}\left(L\rho_t + \rho_t L^{\dag} - tr\big((L + L^{\dag})\rho_t\big)\rho_t\right)\mathrm{d}W_t.
\end{align*}
Here ${H}$ and $\hat{H}$ represent respectively the free and controlled Hamiltonian matrices. The matrix $L$ is the measurement operator associated to the probe.
The observation process of the probe $Y$ is a continuous semimartingale with $\mathrm{d}Y_t = \mathrm{d}W_t + \sqrt{\eta}\,tr\big((L+ L^{\dag})\rho_t\big)\mathrm{d}t,$ where $W$ is a classical Wiener process.  Here $u $ denotes the feedback controller adapted to $\mathcal{F}^{Y_t}$ and
$\eta \in (0,1]$ represents the efficiency of the detector.

\begin{remark}
In the absence of control input and measurement, the dynamics is described by a deterministic linear master equation, called Lindblad master equation.
\end{remark}

\section{$N$-quantum 
particle system and mean-field limit}
\subsection{Belavkin equation for a controlled $N$-particle system}

 Now we consider a system of $N$-quantum particles, where each particle interacts with the  others through an interaction Hamiltonian denoted by ${A}$. Each particle is measured indirectly through an appropriate observable, i.e., $N$-quantum channels are considered. To each particle, a feedback control is applied to achieve certain goals like minimizing a  cost function, maximizing a reward, stabilizing the system, etc. Under our setting, 
${A}$ is given as a symmetric self-adjoint integral operator with Hilbert-Schmidt kernel, i.e. ${A}:\Xc^4\to\C$ is so that 
${A}(l,l';k,k') = {A}(l',l;k',k)$,  
${A}(l,l';k,k') = \overline{{{A}(l,l';k,k')}}.$ 
\begin{align*}
&A: L^2(\Xc^2) \to L^2(\Xc^2).\\
&Af(l,l') := \sum_{(v,v') \in \Xc^2} A(l,l';v,v')f(v,v').
\end{align*}   
By setting $O:=A$ as in Preliminaries, we define similarly $\mathbf{A}_{jk}$.




Hence, the dynamics of the system, identified by the density matrix $\boldsymbol{\rho}^N$, satisfies the Belavkin equation
\begin{align}
\mathrm{d}\boldsymbol{\rho}_t^{N} &\!\!=\! -\iu[\mathbf{H},\boldsymbol{\rho}_t^{N}]\mathrm{d}t + \sum_{j=1}^N \left({\bf L}_j\boldsymbol{\rho}_t^{N}{\bf L}_j^{\dag} - \frac{1}{2}\{{\bf L}_j^{\dag}{\bf L}_j,\boldsymbol{\rho}_t^{N}\}\right)\mathrm{d}t \nonumber \\
&\!\!\!\!\!\!\!+\!\sqrt{\eta}\sum_{j=1}^N\left(\boldsymbol{\rho}_t^{N}{\bf L}^\dag_j+{\bf L}_j\boldsymbol{\rho}_t^{N} - \mathrm{tr}\left(({\bf L}_j+{\bf L}_j^{\dag})\boldsymbol\rho_t^{N}\right)\boldsymbol\rho_t^{N}\right)\mathrm{d}W_t^{j},
\label{eq:particle}
\end{align}
where $\boldsymbol{\rho}_0^{N}=\rho_0^{\otimes N}$, $\mathbf{H}:=\sum_{j} (\mathbf{H}_{j} + u(\rho_t^{j})\mathbf{\hat{H}}_j )+ \sum_{i<j}\mathbf{A}_{ij}/N$, where $\rho_t^{j}$ represents the state of the particle $j$ (for $j=1,\cdots,N$), which can be obtained by taking a partial trace over the other particles.  The corresponding observation process $Y^j$ for particle $j$ is given by
$$\mathrm{d}Y_t^{j} = \mathrm{d}W_t^{j} + \sqrt{\eta}tr\big( ({\bf L}_j + {\bf L}_j^{\dag})\rho_t^{j}\big)\mathrm{d}t.$$
Here we note that  Equation \eqref{eq:particle} is well posed by using similar arguments as in \cite[Propositions 3.3 and 3.5]{mirrahimiHandel07}.
\subsection{Mean-field limit }

As in classical case, we expect that for an appropriate interaction Hamiltonian, when $N$ is large, each particle interacts with an MF instead of interacting individually with the others, and a typical behavior for particles emerges.
For a closed quantum system described by the Schr\"odinger equation, the MF dynamics is given by the well-known Schr\"odinger-Hartee equation \cite{lewin2014derivation,pickl11simple}, and Lindblad-Hartee equation for open quantum systems \cite{merkil2012}. Later, this equation is extended by Kolokoltsov \cite{kolokoltsov2021law,kolokoltsov2022qmfg,kolokoltsov2021qmfgcounting} to treat the case of an open quantum system undergoing continuous-time measurements.

In the following, we consider the later treatment  and we assume in addition that measurements are not perfect, inspired by the previous treatment, we recover formally the following Belavkin equation of MF type
\begin{align}
\mathrm{d}\gamma_t&=  
(-\iu[ {H} + u(\gamma_t){\hat{H}} + {A}^{{m}_t}, \gamma_t] )\mathrm{d}t\nonumber\\
&+ \left(L\gamma_tL^{\dag} - \frac{1}{2}
\{L^{\dag}L,\gamma_t\}\right)\mathrm{d}t\nonumber\\ 
&+\sqrt{\eta}\Big(\gamma_tL^{\dag} + L\gamma_t - tr\big((L + L^{\dag})\gamma_t\big)\gamma_t\Big)\mathrm{d}W_t,
\label{mfbn}
\end{align}
where  $m_t := \mathbb{E}[\gamma_t]$, $\gamma_0 =\rho_0 \in {S}_{d},$ and $A^m=\sum_{\Xc^2}A(l,l';k,k')\overline{m(k,k')}.$
\begin{remark}
   In the absence of control, by taking  expectation, a new nonlinear equation of Lindblad version  can be obtained as follows
\begin{align*}
    \mathrm{d}{m}_t = &-\iu[ {H} + {A}^{m_t} , {m}_t]\mathrm{d}t + \Big(L{m}_tL^{\dag} - \frac{1}{2}\{LL^{\dag},m_t\}\Big)\mathrm{d}t.
\end{align*} 
\end{remark}
\section{Main result }
To justify the above approximation in the MF limit, we have to show that $\boldsymbol\rho^N$ asymptotically becomes close to $\gamma^{\otimes {N}}.$ To measure a deviation from $\boldsymbol\rho^N$ to $\gamma^{\otimes {N}},$  we take  the following quantity considered by Pickl in \cite{pickl11simple}
\begin{align}
\alpha_{N,j}(t) = 1 - tr\big(\gamma_t{\rho}_t^j\big)
= 1 - tr(\boldsymbol\gamma^j_t\boldsymbol\rho^N_t\big),
\label{deviationf}
\end{align}
which is calculated only for the particle $j$ and we recall that  
${\rho}_t^{j}$ corresponds to the partial trace of $\boldsymbol\rho^N$ with respect to the particles other than the particle $j$. Here ${\boldsymbol\gamma^j}:=I\otimes\cdots\otimes \gamma\otimes\cdots\otimes I,$  where $j$-th component of ${\boldsymbol\gamma^j}$ is identified with $\gamma$ and the other components are all identity operator $I.$ 

For the sake of simplicity, we denote $\alpha_N := \alpha_{N,j}$ for any fixed $j.$ By an inequality obtained in \cite[Proposition A.1]{kolokoltsov2022qmfg}, it is sufficient to control $\mathbb{E}[\alpha_N(t)]$ by $\alpha_N(0)$ to state a propagation of chaos result.

In the following theorem, we state the main result of this paper concerning the well-posedness of Equation \eqref{mfbn} (existence and uniqueness of the solution) and propagation of chaos.

\begin{theorem}[well-posedness and propagation of chaos]
Let $T > 0,\; U > 0$, and let $u: {S}_{d} \to [-U,U]$ be bounded and Lipschitz, i.e. 
$|u(\rho) - u(\rho') | \leq \kappa\, \|\rho-\rho'\|$, with $\kappa > 0.$
Then \eqref{mfbn} is well posed and valued in ${S}_{d}$.

Furthermore for $\eta = 1$,  there exists a constant $c \equiv c( ||{A}|| ,||{\hat{H}}||, ||L||)$ such that
\begin{align*} \mathbb{E}\big[ \alpha_{N}(t) \big] &\leq e^{ct}\left(\alpha_N(0)+\frac{1}{\sqrt{N}} \right),
\end{align*}
where $||\cdot||$ corresponds to any matrix norm. In particular, the propagation of chaos is verified under purification assumption.
\label{thm:chaos}
\end{theorem}
\begin{proof}
\paragraph{Well-posedness} The proof will be a combination of arguments in \cite{mirrahimiHandel07} and  \cite[Pages 235-237]{carmona2018mfg1}. For each $\xi\in C\big([0,T], {S}_{d}\big),$ we consider the following equation 
\begin{align}\label{NBol}
    \mathrm{d}{\gamma}_t^{\xi} &= -\iu[ {F}_t^\xi ,{\gamma}_t^{\xi}]\mathrm{d}t + \Big(L{\gamma}_t^{\xi}L^{\dag} - \frac{1}{2}\big\{L^{\dag}L,{\gamma}_t^{\xi}\big\}\Big)\mathrm{d}t\nonumber\\ &+ \sqrt{\eta}\Big(\gamma_t^{\xi}L^{\dag} + L\gamma_t^{\xi} - tr\big((L + L^{\dag})\gamma_t^{\xi}\big)\gamma_t^{\xi}\Big)\mathrm{d}W^{}_t,
\end{align}
where $F_t^\xi:={H} + u(\gamma_t^\xi){\hat{H}} + {A}^{{\xi_t}}.$ This is well posed by similar arguments applied in \cite[Propositions 3.3 and 3.5]{mirrahimiHandel07}. From the existence of the family of equations parametrized by $\xi$, we define the following mapping $\Xi:C\big([0,T], {S}_{d}\big)\to C\big([0,T], {S}_{d}\big)$ by $\Xi(\xi):=(\mathbb{E}[\gamma_t^{\xi}])_{0 \leq t \leq T}$.
Therefore the process $\gamma^m$ corresponds to the  solution of \eqref{mfbn} if and only if $ m=\Xi({m})$. So we should prove the existence and uniqueness  by showing  that the mapping $\Xi$ has a unique fixed point.

To show this, we prove that the map $\Xi$ is a contraction with respect to the uniform norm on $C\big([0,T], {S}_{d}\big).$ Let us pick two arbitrary elements $\xi^1$ and $\xi^2$ in $C\big([0,T], {S}_{d}\big).$ Set $\Delta\gamma_t := \gamma_t^{\xi^1} - \gamma_t^{\xi^2},$
     $\Delta{\xi}_t := \xi^1_t-\xi^2_t,$ and 
$K:=L + L^{\dag}.$
    Then it follows that 
{\small{\begin{align*}
    &\Delta\gamma_t = \int_{0}^{t}\Bigg(-\iu[ {F}^{\xi^1}_s , \Delta \gamma_s] + \Big(L\Delta \gamma_sL^{\dag} - \frac{1}{2}\big\{L^{\dag}L,\Delta\gamma_s\big\}\Big)\Bigg)\mathrm{d}s\\ &+\sqrt{\eta}\int_{0}^{t}\Big(\Delta\gamma_sL^{\dag} + L\Delta\gamma_s\Big)\mathrm{d}W_s -\int_{0}^{t}\big(\iu[{F}^{\xi^{1}}_s-{F}^{\xi^{2}}_s,\gamma_s^{\xi^2}] \big)\mathrm{d}s\\
    &-\sqrt{\eta}\int_{0}^{t}\Big(tr\big(K\gamma_s^{\xi^1}\big)\gamma_s^{\xi^1} - tr\big(K\gamma_s^{\xi^2}\big)\gamma_s^{\xi^2}\Big)\mathrm{d}W_s,
\end{align*}}}%
which yields to the following by H\"older inequality and Itô's isometry
\begin{align*}
   &\|\Xi(\xi^1)_t-  \Xi(\xi^2)_t\| \le  \mathbb{E}\big[\|\Delta\gamma_t\|\big] \\ &\leq\int_{0}^{t}\E\left[\big\|[ {F}^{\xi^1}_s , \Delta \gamma_s]\big\| + \Big\|L\Delta \gamma_sL^{\dag}\Big\| + \frac{1}{2}\Big\|\big\{L^{\dag}L,\Delta\gamma_s\big\}\Big\|\right]\mathrm{d}s\\  
   &+\int_{0}^{t}\E\Big[\Big\|[{F}^{\xi^{1}}_s-{F}^{\xi^{1}}_s,\gamma_s^{\xi^2}]\Big\| \Big]\mathrm{d}s \\
   &+\int_{0}^{t}\left(\E\Big[\Big \|\Delta\gamma_sL^{\dag} + L\Delta\gamma_s\Big\|^2\Big]\right)^{1/2}\mathrm{d}s\\
   & +\int_{0}^{t}\E\Big[\Big \|tr\big(K\gamma_s^{\xi^1}\big)\gamma_s^{\xi^1} - (tr\big(K\gamma_s^{\xi^2}\big)\gamma_s^{\xi^2}\Big \|^2\Big]^{1/2}\mathrm{d}s\\
   &\le \int_{0}^{t} C\mathbb{E}\big[\|\Delta\gamma_s\|+\|\Delta\xi_s\|\big] \mathrm{d}s,  
   \end{align*}
   where  $C>0$ is some constant depending on $T, ||{H}||, ||{A}||, \kappa, \eta, ||{\hat{H}}||, ||L||$. 
In view of Gronwall's inequality, one concludes the existence of some constant, still denoted by $C$ without any danger of confusion
$$\max_{0\le r\le t}\|\Xi(\xi^1)_r-  \Xi(\xi^2)_r\|\le C\int_{0}^{t}\|\Delta \xi_s\|{\mathrm{d}s},\quad \forall t\le T.$$

Replacing $\xi^i$ by $\Xi(\xi^i)$ for $i=1,2$, it follows that  
\begin{align*}\|\Xi^{(2)}(\xi^1)_t - \Xi^{(2)}(\xi^2)_t\| &\leq C^2  \int_{0}^{t}\left(\int_{0}^{s} \|\Delta \xi_r\|\mathrm{d}r\right) \mathrm{d}s\\
&\leq C^2t^2\max_{0\le r\le t}\|\Delta \xi_r\|.
\end{align*}
Repeating the above reasoning, one has for any $k\ge1 $ 
$$\|\Xi^{(k)}(\xi^1)_t - \Xi^{(k)}(\xi^2)_t\| \leq  \frac{C^kt^{k-1}}{(k-1)!}\max_{0\le r\le t}\|\Delta \xi_r\|,$$
where $\Xi^{(k)}$ denotes the $k-$composition of $\Xi$. So for $k$ large enough $\Xi^{(k)}$ is a  contraction. To show the  uniqueness, it is sufficient to pick two  arbitrary solutions $m^{1}, m^{2}.$ We have
\begin{align*}
    \|m^{1}_t - m^{2}_t\| = \|\Xi^{(k)}(m^{1})_t - \Xi^{(k)}(m^{2})_t \|
    \leq c\|m^{1}_t - m^{2}_t\|,
\end{align*}
where $c<1$ denotes some constant. Hence $m^1=m^2.$
\paragraph{Propagation of chaos } 
Here our objective is to estimate   the mean of deviation defined in Equation \eqref{deviationf} by  an inequality depending on the deviation in initial time. 

In order to do this, it's sufficient to estimate this quantity for one of the particles, for instance here we consider the $j$-th particle. For the sake of simplicity, we adapt our notations as follows: 
$\boldsymbol{\gamma} := \boldsymbol{\gamma}^{j}$, $\mathbf{L} := \mathbf{L}_j.$ By  Itô's formula,  we get 
    \begin{align*} \mathrm{d}\alpha_N^{}(t) &= -tr\big(\mathrm{d}\boldsymbol\rho_t^{N}\boldsymbol\gamma_t^{}\big) -tr\big(\boldsymbol{\rho}_t^{N}\mathrm{d}\boldsymbol\gamma_t^{}\big) - tr\big(\mathrm{d}\boldsymbol\rho_t^{N}\mathrm{d}\boldsymbol\gamma_t^{}\big).\\  &= \big(P_t^{(1)} + 
     P_t^{(2)})\mathrm{d}t + \sum_k P_t^{(3,k)}\mathrm{d}W_t^{k},
    \end{align*} 
    \vspace{-6mm}
    
where,
{\small\begin{align*}
P_t^{(1)} &=
\iu tr\Big(\big[\frac{1}{N}\sum_{k \neq j}\mathbf{A}_{kj} - \mathbf{A}_j^{m_t} + \big(u(\rho_t^{j}) - u(\gamma_t)\big)\mathbf{\hat{H}}, \mathbf{I} - \boldsymbol{\gamma}_t\big]\boldsymbol\rho_t^{N}\Big)\\
P_t^{(2)} &\!=\! - tr\big(\boldsymbol{\gamma}_t \mathbf{L}\boldsymbol\rho^N_t\mathbf{L}^{\dag} \!+ \boldsymbol{\gamma}_t \mathbf{L}^{\dag}\boldsymbol{\rho}^N_t \mathbf{L} + \boldsymbol{\gamma}_t \mathbf{L}^{\dag}\boldsymbol{\rho}^{N}_t\mathbf{L}^{\dag} + \boldsymbol{\gamma}_t \mathbf{L} + \boldsymbol{\rho}^{N}_t\mathbf{L}\big) \\ 
&+\Big[tr\big(\boldsymbol{\gamma}_t\boldsymbol\rho^N_t\mathbf{L}^{\dag} + \boldsymbol{\gamma}_t \mathbf{L}\boldsymbol\rho^N_t\big)tr\big(\boldsymbol{\gamma}_t(\mathbf{L}^{\dag}+\mathbf{L})\big) + \\
&tr\big(\boldsymbol{\gamma}_t\boldsymbol\rho^N_t\mathbf{L}^{\dag} + \boldsymbol{\gamma}_t L\boldsymbol\rho^N_t\big)tr\big(\boldsymbol\rho^N_t(\mathbf{L}^{\dag}+\mathbf{L})\big) - \\
&tr\big(\boldsymbol\rho^N_t\boldsymbol{\gamma}_t\big)tr\big(\boldsymbol\rho^N_t(\mathbf{L}^{\dag} + \mathbf{L}\big)\Big)tr\big(\boldsymbol{\gamma}_t(\mathbf{L}+\mathbf{L}^{\dag})\big)\Big],
\end{align*}}%
and $P_t^{(3,k)}$ are bounded quantities. By taking an expectation of the above equation, it follows from the proof of \cite[Theorem 3.1]{kolokoltsov2022qmfg} that there exists $C>0$ such that
{\small\begin{align*}
\frac{\d\E[\alpha_N(t)]}{\d t}
&=\mathbb{E}\big[|P_t^{(1)}|\big]+\mathbb{E}\big[|P_t^{(2)}|\big]  \\
&\le \big(C ||{A}||+\kappa||{\hat{H}}||\big)\mathbb{E}[\alpha_N(t)] + \frac{C}{\sqrt{N}} +\mathbb{E}\big[|P_t^{(2)}|\big].
\end{align*}}%
As for $P^{(2)}_t,$ we combine Lemma \ref{lem1} and \cite[Inequality (44) of Lemma 1]{kolokoltsov2021law}, and obtain $|P_t^{(2)}| \leq C'||L||^2\alpha_N(t)$ for some $C'>0$. Therefore, the proof is fulfilled by Gronwall inequality. 
\end{proof}
It remains to prove Lemma \ref{lem1} which 
 generalizes \cite[Inequality (43) of Lemma 1]{kolokoltsov2021law} and proves 
 \eqref{inemt} without assuming that  $\gamma$ is a one-dimensional projector. 
\begin{lemma}\label{lem1}
Let $A, B$ be in $S_d,$ and $L$ be a hermitian matrix. Then 
\begin{align} &\Bigg|tr(L A LB) - \frac{1}{2}tr(B(L A + A L))tr(BL + A L) \label{inemt}\\
&\!\!\!\!+ tr(B A)tr(BL)tr(A L) \Bigg| \leq 18||L||^2tr\big((I-A)B\big).\nonumber
\end{align}
\end{lemma}
\begin{proof}
Without loss of generality, we consider a basis where $A$ is  diagonal so that we may rewrite $ A= \sum_{k}A_{kk} I_{kk},$ \text{where } $\sum_{k}A_{kk} = 1$ and $I_{kk}$ is the matrix whose only non-zero element is one on the $k$-th diagonal element. Hence,
\begin{align*}\alpha := tr\big( (I -A)B\big) &= tr\big( (I-B)A \big) = \sum_{k}A_{kk}\alpha_k,
\end{align*}
with $\alpha_k := tr( (I - A_{kk})B)$. By the positivity of $B$, it follows that (see \cite{kolokoltsov2021law} for further details) 
\begin{equation} |B_{jr}| \leq \alpha_k, \; j,r \neq k, \quad \max\big(|B_{jk}|,|B_{kj}|\big) \leq \sqrt{\alpha_k}, \; j \neq k 
\label{inl}\end{equation}
Rewrite the left hand side of Inequality \eqref{inemt} as follows:
{\small\begin{align*}
&\Bigg| \sum_{{k}}A_{{kk}}tr(LI_{{kk}}LB) +tr(BL)tr(B A)tr(A L) \\ &-\frac{1}{2}\sum_{{k}}A_{{kk}}tr\Big(B(LI_{{kk}} + I_{{kk}}L)\Big)tr\Big(BL + \sum_{{j}}A_{{jj}}I_{{jj}}L\Big)\Bigg|\\
&\leq \sum_{{k,j}}A_{{kk}}A_{{jj}}\Bigg|B_{{kk}}L_{{jj}}tr(BL) +(LBL)_{{kk}} \\
&- \frac{1}{2}\Big((BL)_{{kk}} + (LB)_{{kk}} \Big)\Big( tr(BL) + L_{{jj}} \Big) \Bigg|\\
&{\leq \sum_{{k}}5||L||^2A_{kk}\alpha_{k}}+\sum_{{k\neq j}}A_{{kk}}A_{{jj}}\Bigg|B_{{kk}}L_{{jj}}tr(BL) +(LBL)_{{kk}} \\
&- \frac{1}{2}\Big((BL)_{{kk}} + (LB)_{{kk}} \Big)\Big( tr(BL) + L_{{jj}} \Big) \Bigg|,
\end{align*}}%
where the second inequality follows from \eqref{inl}. By adding and subtracting $B_{jj}, L_{kk},$ we deduce further
{\small\begin{align*}
&\leq 5||L||^2\alpha+ \sum_{{j} \neq {k}}A_{{kk}}A_{{jj}} 5||L||^2\alpha_k \\&+ \sum_{{j} \neq {k}}A_{{kk}}A_{{jj}}\Bigg|[(B_{{kk}} - B_{{jj}})]L_{{kk}}tr(BL)\Bigg|\\&+   \Bigg|\sum_{{j} \neq {k}}(L_{kk} - L_{jj})A_{kk}A_{{jj}}\Big((BL)_{{kk}} + (LB)_{{kk}}\Big)  \Bigg| \\
 &\leq 12||L||^2 \alpha +   \Bigg|\sum_{{j} \neq {k}}(L_{kk} - L_{jj})A_{kk}A_{{jj}}\Big((BL)_{{kk}} + (LB)_{{kk}}\Big)  \Bigg|
\\&\leq  12||L||^2 \alpha\\
&+\Bigg|\sum_{k}\sum_{j \neq k}A_{jj}A_{kk}(L_{kk} - L_{jj})\sum_{r \neq j,k}\Big[B_{kr}L_{rk} + L_{kr}B_{rk}\Big]\Bigg|\\ &+ \Bigg|\sum_{k}\sum_{j \neq 
 k}(L_{kk} - L_{jj})A_{jj}A_{kk}\Big[B_{kk}L_{kk} + L_{kk}B_{kk}\Big]\Bigg|\\ &+ \Bigg|\sum_{{k}}\sum_{{j} \neq {k}}A_{{jj}}A_{{kk}}(L_{kk} - L_{jj})\Big[B_{{k}{j}}L_{{j}{k}} + L_{{k}{j}}B_{{j}{k}}\Big]\Bigg|\\
& \leq 12||L||^2\alpha+4||L||^2\alpha+2||L||^2\alpha,
 \end{align*}}%
where we apply Fubini and triangular inequality for the second and third terms and use the fact that the last term is equal to zero by symmetry. So the lemma is proved. 
\end{proof}


\section{Applications and Numerical illustration}
In the following section, we consider $N$-quantum particles undergoing imperfect quantum non-demolition measurements, where asymptotically the system converges to pure states which correspond to the equilibria of the system, this phenomenon is known as quantum state reduction, see e.g., \cite{bauer2013repeated,liang2019exponential,handel05red}. Here we  derive the MF Belavkin equation and consider its application in a feedback stabilizing such a system, as propagation of chaos is proved by Theorem \ref{thm:chaos} for perfect measurement under purification assumption, intuitively this motivates our study. This study mathematically is true if the propagation of chaos is valid for imperfect measurement.
Here we focus on measurement-based feedback control strategies, see e.g., \cite{mirrahimiHandel07,qibo10} for a mathematical model description.

\subsection{$N$-quantum particles}
We consider the case of $N$-qubit system (i.e $\mathcal{X} = \{1,2\}$), interacting through a Hamiltonian of MF type. Let the interaction operator between qubits be an operator describing the  exchange of photons \cite[Discussion.6
]{kolokoltsov2021qmfgcounting}, \cite[Remark 8]{kolokoltsov2022dynamic}, where $A = a_1^{\dag}a_2 + a_2^{\dag}a_1$. This operator represents the exchange of a single photon between two qubits, where $a_{j}^{\dag}$ and $a_{j}$ are the creation and annihilation  operators respectively for the $j$-th qubit. The first term describes the process where a photon is absorbed by the first qubit, while the second qubit emits a photon. The second term describes the opposite process. This interaction is given by the tensor $A(l, l'; k, k')$ such that $A(2, 1; 1, 2) = A(1, 2; 2, 1) = 1$ and zeros otherwise. For each particle we associate a free Hamiltonian $\mathbf{H}_j = \boldsymbol{\sigma_z}^{j}$, an  observation channel $\mathbf{L}_j = \boldsymbol{\sigma_z}^{j} $ and a controlled Hamiltonian $\mathbf{\hat{H}}_j = \boldsymbol{\sigma_x}^{j}$. The evolution of the $N$-particles  is given by the following equation:
\vspace{-6mm}

{\small\begin{align*}
\mathrm{d}\boldsymbol{\rho}_t^{N} =& -\iu [\mathbf{H}_t,\boldsymbol{\rho}_t^{N}]\mathrm{d}t + \sum_{j=1}^N \Big(\boldsymbol{\sigma_z}^{j}\boldsymbol{\rho}_t^{N}\boldsymbol{\sigma_z}^{j} - \boldsymbol{\rho}_t^{N}\Big)\mathrm{d}t\\
&+\sqrt{\eta}\sum_{j=1}^N\Big(\boldsymbol{\rho}_t^{N}\boldsymbol{\sigma_z}^{j}+\boldsymbol{\sigma_z}^{j}\boldsymbol{\rho}_t^{N} - 2tr\big(\boldsymbol{\sigma_z}^{j}\boldsymbol\rho_t^{N}\big)\boldsymbol\rho_t^{N}\Big)\mathrm{d}W_t^{j}.
\end{align*}}
\vspace{-4mm}

Note that the simulation of $\boldsymbol\rho^{N}$ requires  $4^N - 1$ real stochastic differential equations and the complexity is $O(4^N)$. 
 
\subsection{Feedback stabilization based on MF Belavkin equation}
Here we aim to make the feedback control depend on  the MF limit equation $\gamma$ instead of the original filter equation $\boldsymbol\rho^N$ and to control $\gamma$ in situation where the interaction between particles is of MF type, by doing this the complexity of the problem can be reduced notably as it is sufficient to  control only the MF particle toward a target state. 
Here we study numerically the asymptotic behavior for MF Belavkin equation.

For the MF equation, the free Hamiltonian will be $H = {\sigma_z}$, the observation channel is $L = {\sigma_z} $, and controlled Hamiltonian is $\hat{H} = {\sigma_x}$.
Straightforward calculations in Pauli basis give us 
$$\scriptsize A^{m} = \begin{pmatrix}
0 & \mathbb{E}[x] - \iu \mathbb{E}[y] \\ \mathbb{E}[x] + \iu \mathbb{E}[y] & 0 
\end{pmatrix}.$$
So the MF Belavkin equation projected in Pauli basis is represented as follows:

{\small\begin{align*}
\small\mathrm{d}x_t &= \small \Big( - y_t - x_t + z_t\mathbb{E}[y_t]\Big)\mathrm{d}t - \sqrt{\eta}x_tz_t\mathrm{d}W_t,\\
\small\mathrm{d}y_t &= \small\Big( x_t - {y_t} + u(\gamma_t)z_t -z_t\mathbb{E}[x_t]\Big)\mathrm{d}t + \sqrt{\eta}y_tz_t\mathrm{d}W_t,\\
\small\mathrm{d}z_t &= \small\Big(-u(\gamma_t)x_t + y_t\mathbb{E}[x_t] + x_t\mathbb{E}[y_t]\Big)\mathrm{d}t + \sqrt{\eta}\big(1 - z_t^2)\mathrm{d}W_t. 
\end{align*}}
To simulate the MF equation, we need to solve only three real stochastic differential equations. Nevertheless, we need to approximate $\mathbb{E}[x_t], \mathbb{E}[y_t], \mathbb{E}[z_t]$ using an $N$-particle system, which yields a complexity $O(N)$.
{\small\begin{align*}
\mathrm{d}x_t^{j} &= \small \Big( - y^{j}_t - x^{j}_t + z^{j}_t\frac{1}{N}\sum_{k=1}^{N}\delta_{y^{k}_t}\Big)\mathrm{d}t - \sqrt{\eta}x^j_tz^j_t\mathrm{d}W^{j}_t,\\
\mathrm{d}y^{j}_t &= \small\Big( x^{j}_t - {y^{j}_t} + u(\gamma_t^{j})z^{j}_t -z^{j}_t\frac{1}{N}\sum_{k=1}^{N}\delta_{x^{k}_t}\Big)\mathrm{d}t + \sqrt{\eta}y^{j}_tz^{j}_t\mathrm{d}W^{j}_t,\\
\mathrm{d}z^{j}_t &\!\!=\!\! \Big(-u(\gamma_t^j)x^{j}_t + \sum_{k=1}^{N}\!\Big(\frac{y^{j}_t}{N}\delta_{x^{k}_t} \!\!+ \!\!\frac{x^{j}_t}{N}\delta_{y^{k}_t}\Big)\Big)\mathrm{d}t \!\!+ \!\!\sqrt{\eta}(1 - z^{2}_t)\mathrm{d}W^{j}_t.
\end{align*}}%
Using Euler's discretization scheme, classical results on the propagation of chaos guarantees the convergence, see e.g.  \cite[Pages 129-130]{liu19phd}.

We start by studying the asymptotic behavior of our system when the feedback control is turned off, i.e., $(u\equiv  0).$ Through numerical simulations, we observe a quantum state reduction property, i.e $(\gamma_t)_{t \geq 0}$  converges to one of the eigenstates of $L,$ i.e., $ \{\rho_e,\rho_g\}$ with \begin{align*}\rho_g := \begin{pmatrix}1 & 0 \\ 0 & 0\end{pmatrix}, \; \rho_e := \begin{pmatrix}0 & 0 \\ 0 & 1\end{pmatrix},\end{align*} that are the equilibrium points of the MF  equation (see Fig. \ref{fig:QSRf}). Further, to ensure that the
system attains a prescribed target, for example $\rho_e$, we adapt a feedback law proposed in \cite{liang2018exponential}. Our  feedback control $u$ is given by  $u(\gamma) := -7.6\iu tr\big([\sigma_x,\gamma]\rho_e\big) + 5\big( 1 - tr(\gamma\rho_e)\big).$ Numerical illustration shows that the stabilization is achieved (see Fig. \ref{fig:Stabilization}).

   \begin{figure}[h!]
     \centering
  \includegraphics[width=0.6\linewidth]{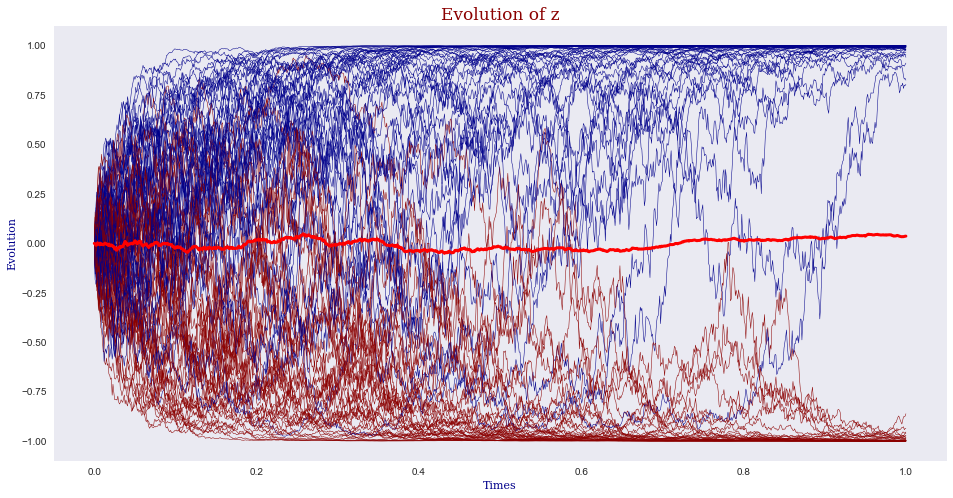}
  \caption{\footnotesize{Evolution of $z$-component for the MF equation. The red curve represents mean trajectories for arbitrary $100$ samples.}}
    \label{fig:QSRf}
  \end{figure}

  \begin{figure}[h!]
    \centering
  \includegraphics[width=0.6\linewidth]{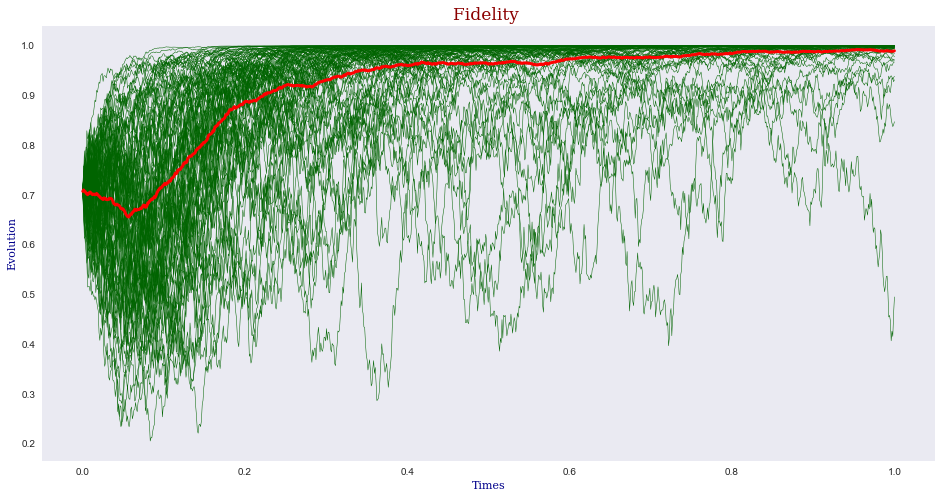}
  \caption{{\footnotesize{Convergence of the fidelity $\mathfrak{F}(\gamma,\rho_e) := \Big(tr\sqrt{\sqrt{\gamma}\rho_e\sqrt{\gamma}} \Big)^2$ toward one with initial state $(x_0,y_0,z_0) = (1/4,-1/4,0) $ and $\eta = 1$. The red curve represents the mean value of 100 arbitrary samples.}}}
  \label{fig:Stabilization}
  \end{figure}

\section{CONCLUSIONS}
In this letter we have considered the works established in
[1], [5]. We show how this framework can be extended to the
case of imperfect measurements. We provide more complete
proof for the well-posedness of the MF equation. We generalize a lemma to estimate the deviation between the quantum
system and its MF approximation, that shows  propagation of
chaos under purification assumption and perfect measurement. Numerical simulations suggest quantum state reduction and
stabilization by applying feedback control toward a target
state. This approximation in MF allows to study control and
stabilization of a system of $N$  continuous monitoring
interacting particles.

In further research we will provide rigorous proof for the
long-time behavior and stabilization for the MF equation. Also the link between such the study and quantum trajectory formalism is interesting to be investigated, as it is already established for Belavkin filtering equation, see e.g., \cite{barchielli1995constructing}. 

\bibliographystyle{unsrt}
\bibliography{ref}

\begin{thebibliography}{10}

\bibitem{kolokoltsov2022qmfg}
V.~N. Kolokoltsov.
\newblock Quantum mean-field games.
\newblock {\em The Annals of Applied Probability}, 32(3):2254--2288, 2022.

\bibitem{huang2006large}
M.~Huang, R.~P. Malham{\'e}, and P.~E. Caines.
\newblock Large population stochastic dynamic games: closed-loop
  {M}ckean-{V}lasov systems and the {N}ash certainty equivalence principle.
\newblock {\em Communications in Information \& Systems}, 6(3):221--252, 2006.

\bibitem{lasry2006mfg1}
J-M. Lasry and P-L. Lions.
\newblock Jeux {\`a} champ moyen. i--le cas stationnaire.
\newblock {\em Comptes Rendus Math{\'e}matique}, 343(9):619--625, 2006.

\bibitem{sudarshan77}
B.~{Misra} and E.~C.~G. {Sudarshan}.
\newblock {The Zeno's paradox in quantum theory}.
\newblock {\em Journal of Mathematical Physics}, 18(4):756--763, April 1977.

\bibitem{kolokoltsov2021law}
V.~N. Kolokoltsov.
\newblock The law of large numbers for quantum stochastic filtering and control
  of many-particle systems.
\newblock {\em Theoretical and Mathematical Physics}, 208(1):937--957, 2021.

\bibitem{kolokoltsov2022dynamic}
V.~N. Kolokoltsov.
\newblock Dynamic quantum games.
\newblock {\em Dynamic Games and Applications}, 12(2):552--573, 2022.

\bibitem{kolokoltsov2021qmfgcounting}
V.~N. Kolokoltsov.
\newblock Quantum mean-field games with theobservations of counting type.
\newblock {\em Games}, 12(1):7, 2021.

\bibitem{pickl11simple}
P.~Pickl.
\newblock A simple derivation of mean field limits for quantum systems.
\newblock {\em Letters in Mathematical Physics}, 97(2):151--164, 2011.

\bibitem{serafini12feedback}
A.~Serafini.
\newblock Feedback control in quantum optics: An overview of experimental
  breakthroughs and areas of application.
\newblock {\em International Scholarly Research Notices}, 2012, 2012.

\bibitem{gough13}
J.~Gough and V.~Belavkin.
\newblock Quantum control and information processing.
\newblock {\em Quantum Information Processing}, 12:1397--1415, 2013.

\bibitem{handel05rev}
R.~Van~Handel, J.K. Stockton, and H.~Mabuchi.
\newblock Modelling and feedback control design for quantum state preparation.
\newblock {\em Journal of Optics B: Quantum and Semiclassical Optics},
  7(10):S179, September 2005.

\bibitem{wiseman2009quantum}
H.~M. Wiseman and G.~J. Milburn.
\newblock {\em Quantum measurement and control}.
\newblock Cambridge university press, 2009.

\bibitem{boutenhandel07}
L.~Bouten, R.~Van~Handel, and M.~R. James.
\newblock An introduction to quantum filtering.
\newblock {\em SIAM J. Control Optim.}, 46(6):2199–2241, dec 2007.

\bibitem{gough22}
J.~Gough.
\newblock Quantum covariance and filtering.
\newblock {\em Annual Reviews in Control}, 54:262--273, 2022.

\bibitem{ohki18}
K.~Ohki.
\newblock An invitation to quantum filtering and smoothing theory based on two
  inner products.
\newblock {\em RIMS Kôkyûroku published}, pages 18--44, 2018.

\bibitem{belavkin01QuantumNB}
V.~Belavkin.
\newblock Quantum noise, bits and jumps: uncertainties, decoherence,
  measurements and filtering.
\newblock {\em Progress in Quantum Electronics}, 25(1):1--53, 2001.

\bibitem{maassen2006purification}
H.~Maassen and B.~K{\"u}mmerer.
\newblock Purification of quantum trajectories.
\newblock {\em Lecture Notes-Monograph Series}, pages 252--261, 2006.

\bibitem{handel05red}
R.~van Handel, J.K. Stockton, and H.~Mabuchi.
\newblock Feedback control of quantum state reduction.
\newblock {\em IEEE Transactions on Automatic Control}, 50(6):768--780, 2005.

\bibitem{bauer2013repeated}
M.~Bauer, T.~Benoist, and D.~Bernard.
\newblock Repeated quantum non-demolition measurements: convergence and
  continuous time limit.
\newblock In {\em Annales Henri Poincar{\'e}}, volume~14, pages 639--679.
  Springer, 2013.

\bibitem{bauer2011convergence}
M.~Bauer and D.~Bernard.
\newblock Convergence of repeated quantum nondemolition measurements and
  wave-function collapse.
\newblock {\em Physical Review A}, 84(4):044103, 2011.

\bibitem{liang2019exponential}
W.~Liang, N.~H. Amini, and P.~Mason.
\newblock On exponential stabilization of n-level quantum angular momentum
  systems.
\newblock {\em SIAM Journal on Control and Optimization}, 57(6):3939--3960,
  2019.

\bibitem{liang2018exponential}
W.~Liang, N.~H. Amini, and P.~Mason.
\newblock On exponential stabilization of spin-1/2 systems.
\newblock In {\em 2018 IEEE Conference on Decision and Control (CDC)}, pages
  6602--6607. IEEE, 2018.

\bibitem{belavkin92}
V.~Belavkin and P.~Staszewski.
\newblock Nondemolition observation of a free quantum particle.
\newblock {\em Physical Review A}, 1992.

\bibitem{mirrahimiHandel07}
M.~Mirrahimi and R.~Van~Handel.
\newblock Stabilizing feedback controls for quantum systems.
\newblock {\em {SIAM} J. Control. Optim.}, 46(2):445--467, 2007.

\bibitem{lewin2014derivation}
M.~Lewin, P.~T. Nam, and N.~Rougerie.
\newblock Derivation of {H}artree's theory for generic mean-field bose systems.
\newblock {\em Advances in Mathematics}, 254:570--621, 2014.

\bibitem{merkil2012}
M.~Merkli and G.~Berman.
\newblock Mean-field evolution of open quantum systems : an exactly solvable
  model.
\newblock {\em Proceedings of the Royal Society A: Mathematical, Physical and
  Engineering Sciences}, 2012.

\bibitem{carmona2018mfg1}
R.~Carmona and F.~Delarue.
\newblock {\em Probabilistic Theory of Mean Field Games with Applications I
  Mean Field FBSDEs, Control, and Games}.
\newblock Probability Theory and Stochastic Modelling. Springer Nature, United
  States, 2018.

\bibitem{qibo10}
B.~Qi and L.~Guo.
\newblock Is measurement-based feedback still better for quantum control
  systems.
\newblock {\em Systems \& Control Letters}, 2010.

\bibitem{liu19phd}
Y.~Liu.
\newblock {\em Optimal Quantization: Limit Theorem, Clustering and Simulation
  of the McKean-Vlasov Equation}.
\newblock PhD thesis, Sorbonne universit{\'e}, 2019.

\bibitem{barchielli1995constructing}
A.~Barchielli and A.S. Holevo.
\newblock Constructing quantum measurement processes via classical stochastic
  calculus.
\newblock {\em Stochastic Processes and their applications}, 58(2):293--317,
  1995.

\end{thebibliography}

\end{document}